\documentclass[a4paper]{amsart}
\usepackage{amssymb}
\usepackage{mathrsfs}                     % Calligraphic font
\usepackage{layouts}
\usepackage{graphicx}
\usepackage[all]{xy}% xymatrix for CD's
\allowdisplaybreaks

\def\undersetbrace#1\to#2{\underbrace{#2}_{#1}}
\def\oversetbrace#1\to#2{\overbrace{#2}^{#1}}
\def\AMSunderset#1\to#2{\underset{#1}{#2}}
\def\AMSoverset#1\to#2{\overset{#1}{#2}}

\swapnumbers

\newtheorem*{proposition*}{Proposition}

\newtheorem*{theorem*}{Theorem}

\newtheorem*{lemma*}{Lemma}

\newtheorem*{corollary*}{Corollary}

\theoremstyle{definition}
\newtheorem*{remark*}{Remark}

\def\ign#1{}             %=ignore, invisible entry for the index only
\def\o{\circ}
\def\X{\mathfrak X}
\def\al{\alpha}
\def\be{\beta}
\def\ga{\gamma}

\def\ep{\varepsilon}

\def\la{\lambda}
\def\rh{\rho}
\def\si{\sigma}
\def\ta{\tau}
\def\ph{\varphi}

\def\ps{\psi}

\def\Ga{\Gamma}

\def\Ph{\Phi}
\def\Ps{\Psi}

\def\i{^{-1}}
\def\inv{^{-1}}
\def\x{\times}
\def\p{\partial}
\let\on=\operatorname

\def\Diff{\on{Diff}}
\def\Dens{\on{Dens}}
\def\Prob{\on{Prob}}
\def\Vol{\on{Vol}}

\def\R{{\mathbb R}}

\begin{document}%\topmatter
\title[Geometry of the Fisher--Rao metric]
{Geometry of the Fisher--Rao metric on the space of smooth densities on a compact manifold}

\author{Martins Bruveris, Peter W. Michor}
\address{
Martins Bruveris: 
Department of Mathematics, 
Brunel University London, Ux\-bridge, UB8 3PH, United Kingdom
\newline
Peter W.\ Michor: Fakult\"at f\"ur Mathematik,
Universit\"at Wien, Os\-kar-Mor\-gen\-stern-Platz 1, A-1090 Wien, Austria.
}

\email{martins.bruveris@brunel.ac.uk}
\email{peter.michor@univie.ac.at}

\date{{\today} } 

\thanks{MB was supported by a BRIEF award from Brunel University London} 

\begin{abstract} It is known that
on a closed manifold of dimension greater than one, every smooth weak Riemannian metric on the space of smooth positive  densities that is invariant under the action of the diffeomorphism group, is of the form 
$$
G_\mu(\al,\be)=C_1(\mu(M)) \int_M \frac{\alpha}{\mu}\frac{\beta}{\mu}\,\mu + C_2(\mu(M))  \int_M\alpha \cdot \int_M\beta
$$
for some smooth functions $C_1,C_2$ of the total volume $\mu(M)$.
Here we determine the geodesics  and the curvature of this metric and study geodesic and metric completeness.
\end{abstract}
\def\LaTeXonly{}%\endtopmatter

\maketitle

\typeout{TEXTWIDTH: \prntlen{\textwidth}}

%\section{Introduction}

\subsection{Introduction}
The Fisher--Rao metric on the space $\on{Prob}(M)$ of probability densities is invariant under the action of the diffeomorphism group $\Diff(M)$.
Restricted to finite-dimensional submanifolds of $\on{Prob}(M)$, 
so-called statistical manifolds, it is called Fisher's information metric \cite{Ama1985}. 
A uniqueness result was established \cite[p. 156]{Cen1982} for Fisher's information 
metric on finite sample spaces and \cite{AJLS2014} extended it to infinite sample spaces.       
The Fisher--Rao metric on the infinite-dimensional manifold of all positive probability densities 
was studied in \cite{Fri1991}, including the computation of its curvature. 
In \cite{BBM2016} it was proved that any $\Diff(M)$-invariant Riemannian metric on the space $\on{Dens}_+(M)$ of smooth positive densities on a compact manifold $M$ without boundary is of the form
\begin{equation}\label{Fisher-Rao-metric} 
G_\mu(\al,\be)=C_1(\mu(M)) \int_M \frac{\alpha}{\mu}\frac{\beta}{\mu}\,\mu + C_2(\mu(M))  \int_M\alpha  \cdot \int_M\beta
\end{equation}
for some smooth functions $C_1,C_2$ of the total volume $\mu(M)$.
This implies that the Fisher--Rao metric on $\Prob(M)$ is, up to a multiplicative constant, the unique $\Diff(M)$-invariant metric. By Cauchy--Schwarz the metric \eqref{Fisher-Rao-metric} is positive definite if and only if 
$C_2(m)>-\frac1m C_1(m)$ for all $m>0$.

\subsection{The setting}
Let $M^m$ be a smooth compact manifold. It may have  boundary or it may even be a manifold with corners; i.e., modelled on open subsets of quadrants in $\mathbb R^m$. 
For a detailed description of the line bundle of smooth densities we refer to \cite{BBM2016} or \cite[10.2]{Mic2008}.
We let $\on{Dens}_+(M)$ denote the space of smooth positive densities on $M$, i.e., 
$\on{Dens}_+(M) = 
\{ \mu \in \Ga(\on{Vol}(M)) \,:\, \mu(x) > 0\; \forall x \in M\}$. Let $\on{Prob}(M)$ be the 
subspace of positive densities with integral 1 on $M$. Both spaces are smooth Fr\'echet manifolds;
in particular they are open subsets of the affine spaces of all densities and densities of integral 
1 respectively. For $\mu \in \on{Dens}_+(M)$ we have    
$T_\mu \on{Dens}_+(M) = \Ga(\on{Vol}(M))$ and for $\mu\in \Prob(M)$ we have 
$$
T_\mu\Prob(M)=\{\al\in \Ga(\on{Vol}(M)): \int_M\al =0\}.
$$
The Fisher--Rao metric, given by 
$G^{\operatorname{FR}}_\mu(\al,\be) = \int_M \frac{\al}{\mu}\frac{\be}{\mu}\mu$ is a Riemannian metric on $\on{Prob}(M)$; it is invariant under the natural action of the group $\Diff(M)$ of all diffeomorphisms of $M$. 
If $M$ is compact without boundary of dimension $\ge 2$, the Fisher-Rao metric is the unique $\Diff(M)$-invariant metric up to a multiplicative constant.
This follows, since any $\Diff(M)$-invariant Riemannian metric on $\on{Dens}_+(M)$ is of the form \eqref{Fisher-Rao-metric} as proved in \cite{BBM2016}.

%\section{Geometry of the Fisher--Rao Metric}

\subsection{Overview}
\label{overview}

We will study four different representations of the metric $G$ in \eqref{Fisher-Rao-metric}.
The first representation is $G$ itself on the space $\on{Dens}_{+}(M)$. 
Next we fix a density $\mu_0 \in \on{Prob}(M)$ and consider the mapping
\[
R:\Dens_+(M)\to C^\infty(M,\mathbb R_{>0})\,, \qquad R(\mu) = f = \sqrt{\frac{\mu}{\mu_{0}}}\,.
\]
This map is a diffeomorphism with inverse $R\i(f)= f^2\mu_0$, and we will denote the induced metric by $\tilde G = \left(R\inv\right)^\ast G$; it is given by the formula
\[
\tilde G_{f}(h,k) = 4 C_1(\| f\|^2) \langle h, k \rangle
+ 4 C_2(\| f \|^2) \langle f, h \rangle \langle f, k \rangle\,,
\]
with $\| f\|^2 = \int_M f^2 \mu_0$ denoting the $L^2(\mu_0)$-norm, and this formula makes sense for $f \in C^{\infty}(M,\mathbb R)$. See Sect.\ \ref{geodesics} for calculations.

Next we take the pre-Hilbert space $(C^{\infty}(M,\mathbb R), \langle\;,\;\rangle_{L^2(\mu_0)})$ and pass to polar coordinates.
Let  
%$S = \{ \ph \in C^\infty(M,\R) \,:\, \int_M \ph^2 \mu_0 = 1\}$ be the sphere in $C^{\infty}$, and let 
$S = \{ \ph \in L^2(M,\R) \,:\, \int_M \ph^2 \mu_0 = 1\}$ denote the $L^2$-sphere. Then
\[
\Ph : C^\infty(M, \R_{>0}) \to \R_{>0} \x (S \cap C^\infty_{>0})\,,\qquad
\Ph(f) = (r,\ph) = \left(\| f \|, \frac f{\|f\|}\right)\,,
\]
is a diffeomorphism, where $C^\infty_{>0} = C^\infty(M,\R_{>0})$; its inverse is $\Ph\i(r,\ph)= r.\ph$.  We set $\bar G = \left(\Ph\inv\right)^\ast \tilde G$; the metric $\bar G$ has the expression
\[
\bar G_{r,\ph} = g_1(r) \langle d\ph, d\ph \rangle + g_2(r) dr^2\,,
\]
with $g_1(r) = 4C_1(r^2) r^2$ and $g_2(r) = 4\left(C_1(r^2) + C_2(r^2) r^2\right)$.
Finally we change the coordinate $r$ diffeomorphically to 
\[
s= W(r) = \int_1^r \sqrt{g_2(\rh)} \,d\rh\,.
\]
Then, defining $a(s) = 4 C_1(r(s)^2) r(s)^2$, we have
\[
\bar G_{s,\ph} = a(s) \langle d\ph, d\ph \rangle + ds^2\,.
\]
We will use $\bar G$ to denote the metric in both $(r,\ph)$ and $(s,\ph)$ coordinates.
Let $W_- = \lim_{r\to 0+} W(r)$ and $W_+ = \lim_{r\to \infty} W(r)$. Then $W: \R_{>0} \to (W_-, W_+)$ is a diffeomorphism. This completes the first row in Fig.~\ref{fig:diagram}. The geodesic equation of $G$ in the various representations will be derived in Sect.~\ref{geodesics}. The formulas for the geodesic equation and later for curvature are infinite-dimensional analoga of the corresponding formulas for warped products; see \cite[p.~204ff]{ONeill1983} or \cite[Chap.~7]{Chen2011}.

The four representations are summarized in the following diagram.
\[
\xymatrix @C-0.5pc { 
\on{Dens}_+(M) \ar[r]^-R  & C^\infty(M,\R_{>0}) \ar[r]^-\Ph &
\R_{>0} \x (S \cap C^\infty_{>0}) \ar[r]^-{W \x \on{Id}} & 
(W_-, W_+) \x (S \cap C^\infty_{>0})\,.
}
\]

% \begin{table}
% \centering
% \begin{tabular}{ccccc}
% $\on{Dens}_+(M)$ & $\on{Dens}(M) \setminus \{0\}$
% & $\Ga_{L^1}(\on{Vol}(M)) \setminus \{0\}$ & $\mu$ & G \\
% $C^\infty(M, \R_{>0})$ & $C^\infty(M,\R) \setminus \{0\}$
% & $L^2(M,\R) \setminus \{0\}$ & $f$ & $\tilde G$ \\
% $\R_{>0} \x S\cap C^{\infty}_{>0}$ & $\R_{>0} \x S\cap C^\infty$ &
% $\R_{>0} \x S$ & $(r,\ph)$ & $\bar G$ \\
% $\R \x S\cap C^{\infty}_{>0}$ & $\R \x S\cap C^\infty$ &
% $\R \x S$ & $(r,\ph)$ & $\bar G$
% \end{tabular}
% \end{table}

\begin{figure}
\[
\xymatrix @C-1pc { 
\xyoption{rotate} 
\on{Dens}_+(M) \ar[r]^-R \ar[d]  & C^\infty(M,\mathbb R_{>0}) \ar[r]^-\Ph \ar[d] &
\mathbb R_{>0} \!\x\! S \cap C^\infty_{>0} \ar[r]^-{W \x \on{Id}} \ar[d] & 
(W_-, W_+) \!\x\! S \cap C^\infty_{>0} \ar[d] \\
\Ga_{C^1}(\on{Vol}(M)) \!\setminus\! \{0\} \ar[d] & 
C^\infty(M,\mathbb R) \!\setminus\! \{0\} \ar[l]_{\;\;R\i} \ar[r]^-\Ph \ar[d] &
\mathbb R_{>0} \!\x\! S \cap C^\infty\ar[r]^-{W \x \on{Id}} \ar[d] & \mathbb R \!\x\! S \cap C^\infty \ar[d] \\
\Ga_{L^1}(\on{Vol}(M)) \!\setminus\! \{0\} \ar[r]^-R & 
L^2(M,\mathbb R) \!\setminus\! \{0\} \ar[r]^-\Ph &
\mathbb R_{>0} \!\x\! S \ar[r]^-{W \x \on{Id}} & \mathbb R \!\x\! S
}
\]
\caption{Representations of $\on{Dens}_+(M)$ and its completions. In the second and third rows we assume that $(W_-, W_+) = (-\infty,+\infty)$ and we note that $R$ is a diffeomorphism only in the first row.}
\label{fig:diagram}
\end{figure}

Since $\bar G$ induces the canonical  metric on $(W_-, W_+)$, a necessary condition for $\bar G$ 
to be geodesically complete is $(W_-, W_+) = (-\infty, +\infty)$. Rewritten in terms of the functions $C_1$ and  $C_2$ this becomes
\[
W_+ = \infty
\Leftrightarrow\left(
\int_1^\infty r^{-1/2} \sqrt{C_1(r)} \,dr = \infty
\text{ or }
\int_1^\infty \sqrt{C_2(r)} \,dr = \infty
\right)\,,
\]
and similarly for $W_-=-\infty$, with the limits of integration being 0 and 1.
If $\bar G$ is incomplete, i.e., $W_->-\infty$ or $W_+<\infty$, there are sometimes geodesic completions. See Sect.~\ref{completions} for details.

We now assume that $(W_-, W_+) = (-\infty, +\infty)$. The metrics $\bar G$ and $\tilde G$ can be extended to the spaces $\R\x S \cap C^\infty$ and $C^\infty(M,\R) \setminus \{0\}$ and the last two maps in the diagram
\[
\xymatrix{
\Ga_{C^1}(\on{Vol}(M)) \!\setminus\! \{0\} & 
C^\infty(M,\R) \!\setminus\! \{0\} \ar[l]_{\qquad R\i} \ar[r]^-\Ph &
\R_{>0} \x S \cap C^\infty \ar[r]^-{W \x \on{Id}}& \R \x S \cap C^\infty
}
\]
are bijections. The extension of $R\inv$ is given by $R\inv(f) = f |f| \mu_0$; it does not map into smooth densities any more, but only into $C^1$-sections of the volume bundle; however, $R\inv$ is not surjective into $C^1$-sections, because the loss of regularity for $R\inv(f)$ occurs only at point where $f$ is $0$.
%The extension of $R$ is given by $R(\mu) = \on{sgn}(\mu) \sqrt{|\mu|/\mu_0}$ and its inverse is $R\inv(f) = f |f| \mu_0$. 
The last two maps, $\Ph$ and $W \x \on{Id}$, are diffeomorphisms.
%, but the first one is not. The extension of $R$ to $\on{Dens}(M) \setminus \{0\}$ is continuous, but not $C^1$, and its inverse is $C^1$, but not $C^2$; furthermore $DR\inv(f)$ is not surjective, if $f(x)=0$ for some $x \in M$. 
The following will be shown in Sect.~\ref{completeness}: $(W_-, W_+) = (-\infty, +\infty)$ implies that $(\R \x S \cap C^\infty, \bar G)$ is geodesically complete and hence so are $(\R_{>0} \x S \cap C^\infty, \bar G)$ and $(C^\infty(M,\R) \setminus \{0\}, \tilde G)$.
%If we want to consider $(\on{Dens}(M) \setminus \{0\}, G)$ as a geodesically complete manifold -- since $R$ is not a diffeomorphism -- we have to use the differential structure obtained from the pullback via $R$; with this differential structure $R$ is by definition  a diffeomorphism.

Finally we consider the metric completions, still assuming that $(W_-, W_+) = (-\infty, +\infty)$. 
For $\bar G$ this is $\R \x S$ or $\R_{>0} \x S$ in $(s,\ph)$ or $(r,\ph)$-coordinates, respectively, as shown in Sect.~\ref{completeness}. The metrics and maps can be extended to
\[
\xymatrix{
\Ga_{L^1}(\on{Vol}(M)) \setminus \{0\} \ar[r]^-R & 
L^2(M,\R) \setminus \{0\} \ar[r]^-\Ph &
\R_{>0} \x S \ar[r]^-{W \x \on{Id}} & \R \x S
}\,.
\]
Here $\Ga_{L^1}$ denotes the space of $L^1$-sections. The extension of $R$ is given by $R(\mu) = \on{sgn}(\mu) \sqrt{|\mu|/\mu_0}$ and its inverse is $R\inv(f) = f |f| \mu_0$ as before. The last two maps are diffeomorphisms and hence $(L^2(M,\R) \setminus \{0\}, \tilde G)$ is metrically complete. The extension of $R$ is bijective, but not a diffeomorphism. It is continuous, but not $C^1$, and its inverse is $C^1$, but not $C^2$; furthermore $DR\inv(f)$ is not surjective if $f=0$ on a set of positive measure. However we can use $R$ to pull back the geodesic distance function from $L^2(M,\R) \setminus \{0\}$ to $\Ga_{L^1}(\on{Vol}(M)) \setminus \{0\}$ to obtain a complete metric on the latter space, that is compatible with the standard topology.

\subsection{The inverse $R\i$ and geodesic completeness}
There is more than one choice for the extension of $R\inv(f) = f^2 \mu_0$ from $C^\infty(M,\R_{>0})$ to $C^\infty(M,\R)$. The choice $R\inv(f) = f |f| \mu_0$ remains injective and can be further extended to a bijection on the metric completion $L^2(M,\R) \setminus \{0\}$. We can consider the equally natural extension $Q$ and its factorizarion given by
\begin{gather*}
\xymatrix{
C^\infty(M, \mathbb R) \ar[rr]^{Q} \ar[dr]_{Q_1} && \Ga_{\ge 0}(\Vol(M))\\
& \{|f|: f\in C^\infty(M,\mathbb R)\} \ar[ur]_{Q_2}&
}
\\
Q(f) = f^2\mu_0\,,\quad Q_1(f) = |f|\,, \quad Q_2(|f|) = |f|^2\mu_0\,.
\end{gather*}
into the space of smooth, nonnegative sections. The map $Q$ is not surjective; see \cite{KLM04} for a discussion of smooth non-negative functions admitting smooth square roots. 

The image $ \{|f|: f\in C^\infty(M,\mathbb R)\}$ of $Q_1$ looks somewhat like the orbit space of a discrete reflection group:
An example of a codimension 1 wall of the image could be $\{|f|: f\in C^\infty(M,\mathbb R), f(x)=0\}$ for one fixed point $x\in M$. Since this is dense in the $L^2$-completion of $T_f C^\infty(M,\mathbb R)$ with respect to $\tilde G_f$, we do not have a reflection at this wall.
Fixing $\ph_0\in S\cap C^\infty$ and considering $\{(r,\ph) \in \mathbb R_{>0}\x S\cap C^\infty: \langle \ph_0, \ph\rangle = 0 \}$ we can write the orthogonal reflection $(r, t_1 \ph_0 + t_2 \ph)\mapsto (r, -t_1\ph_0 + t_2\ph)$.
Geodesics in $(C^\infty(M,\mathbb R),\tilde G)$ are mapped by $Q_1$ to curves that are geodesics in the interior $C^\infty(M,\mathbb R_{>0})$, and that are reflected following Snell's law at any hyperplanes in the boundary for which the angle makes sense. The mapping $Q_2$ then smoothes out the reflection to a `quadratic glancing of the boundary' if one can describe the smooth structure of the boundary. It is tempting to paraphrase this as: \emph{The image of $Q$ is geodesically complete.}
But note that: \thetag{1} The metric $G$ becomes ill-defined on the boundary. \thetag{2} The boundary is very complicated; each closed subset of $M$ is the zeroset of a smooth non-negative function and thus corresponds to a `boundary component'. Some of them  `look like reflection walls'. One could try to set up a theory of infinite dimensional stratified Riemannian manifolds and geodesics on them to capture this notion of geodesic completeness, similarly to \cite{Michor03orbit}. But the situation is quite clear geometrically, and we prefer to consider the geodesic completion described by the inverse $R\i$ used in this paper, which is perhaps more natural.

\subsection{Geodesics of the Fisher-Rao metric on $\Dens_+(M)$}\label{geodesics}
In \cite{Fri1991} it was shown that $\Prob(M)$ has constant sectional curvature for the Fisher-Rao metric.  For fixed 
$\mu_{0}\in \Prob(M)$ we consider the mapping 
$$R:\Dens_+(M)\to C^\infty(M,\mathbb R_{>0}), \qquad R(\mu) = \sqrt{\frac{\mu}{\mu_{0}}}\,.$$
The inverse $R\i: C^{\infty}(M,\mathbb R_{>0})\to \Dens_{+}(M)$ is given by $R\i(f)=f^{2}\mu_{0}$; its tangent mapping is $T_{f}R\i.h = 2fh\mu_{0}$.

\begin{remark*}
In \cite{KLMP2013} it was shown that for $C_1 \equiv 1$ and $C_2\equiv 0$ the rescaled map 
$R(\mu) = 2 \sqrt{\frac{\mu}{\mu_{0}}}$ is an isometric diffeomorphism from $\on{Prob}(M)$ onto the open subset $C^\infty(M, \R_{>0}) \cap \{ f \,:\, \int f^2 \mu_0 = 4\}$ of the $L^2$-sphere of radius $2$ in the pre-Hilbert space 
$(C^{\infty}(M,\mathbb R), \langle\;,\;\rangle_{L^{2}(\mu_{0})})$. For a general function $C_1$ the same holds for 
$R(\mu) = \la \sqrt{\frac{\mu}{\mu_{0}}}$ and the $L^2$-sphere of radius $\la$, where $\la>0$ is a solution of the equation $\la^2 = 4C_1(\la^{-2})$.
\end{remark*}
 
The Fisher--Rao metric 
%as in the main theorem thus 
induces the following metric on the open convex cone 
$C^{\infty}(M,\mathbb R_{>0})\subset C^{\infty}(M,\mathbb R)$:
\begin{multline*}\tag{a}
\left((R\i)^{*}G\right)_{f}(h,k) = G_{R\i(f)}(T_{f}R\i.h, T_{f}R\i.k) = G_{f^2\mu_0}(2fh\mu_0,2fk\mu_0)
\\
= C_1(\|f\|^2_{L^2(\mu_0)}) \int \frac{2fh\mu_0}{f^2\mu_0}\frac{2fk\mu_0}{f^2\mu_0} f^2\mu_0 
+ C_2(\|f\|^2_{L^2(\mu_0)})\int 2fh\mu_0\cdot \int 2fk\mu_0
\\
= 4C_1(\|f\|^2) \int hk\mu_0 + 4C_2(\|f\|^2)\int fh\mu_0\cdot \int fk\mu_0
\\= 4C_1(\|f\|^2) \langle h,k\rangle  + 4C_2(\|f\|^2)\langle f,h\rangle \langle f,k\rangle
\\
\shoveleft{= 4C_1(\|f\|^2) \Big\langle h-\frac{\langle f,h\rangle}{\|f\|^2} f,k-\frac{\langle f,k\rangle}{\|f\|^2} f\Big\rangle\; +} 
\\ 
+ 4\big(C_2(\|f\|^2).\|f\|^2 + C_1(\|f\|^2)\big)\Big\langle \frac{f}{\|f\|},h\Big\rangle \Big\langle \frac{f}{\|f\|},k\Big\rangle\,,
\end{multline*}
where in the last expression we split $h$ and $k$ into the parts  perpendicular to $f$ and multiples of $f$. 

We now switch to 
\emph{polar coordinates on the pre-Hilbert space}: Let $S = \{\ph\in L^2(M,\R): \int \ph^2\mu_0 =1 \}$ denote the sphere, and let $S\cap C^{\infty}_{>0}$ be the intersection with the positive cone. Then 
$C^{\infty}(M,\mathbb R) \setminus \{0\}\cong \mathbb R_{>0}\x S\cap C^{\infty}$ via 
\[
\Ph : C^\infty(M, \R) \setminus \{0\} \to \R_{>0} \x S\,,\qquad
\Ph(f) = (r,\ph) = \left(\| f \|, \frac f{\|f\|}\right)\,.
\]
Note that $\Ph(C^{\infty}(M,\mathbb R_{>0}))=\mathbb R_{>0}\x S\cap C^{\infty}_{>0}$.
We have $f = \Ph\i(r,\ph) = r.\ph$
thus $df = r\,d\ph +\ph\,dr $, 
%and $T_{(r,\ph)}\Ph\i.(\rh,\ps) = r\ps + \rh \ph$. Note that $\Ph\i$ maps $\R_{>0} \x S^+$ diffeomorphically onto $C^\infty(M, \R_{>0})$.
where $r\,d\ph(h)= h- \langle\ph,h\rangle\ph$ is the orthogonal projection onto the tangent space of $S$ at $\ph$ and $dr(h)=\langle\ph,h\rangle$. The Euclidean (pre-Hilbert) metric  in polar coordinates is given by
 \begin{align*}
 \langle df,df \rangle &= \langle \ph.dr + r.d\ph,\ph.dr + r.d\ph\rangle 
 = \langle\ph,\ph\rangle dr^2 + 2r.\langle\ph,d\ph\rangle.dr + r^2\langle d\ph,d\ph\rangle
 \\&
 = dr^2 + r^2\langle d\ph,d\ph\rangle\,.
 \end{align*}
%Denote for now $\tilde G = (R\i)^\ast G$ and calculate the pullback
%\begin{align*}\tag{b}
%\left((\Ph\i)^\ast \tilde G\right)_{(r,\ph)}\left((\rh, \ps), (\rh, \ps)\right)
%&= \left((R\i \circ \Ph\i)^\ast G\right)_{(r,\ph)} \left((\rh, \ps), (\rh, \ps)\right) \\
%&= 4C_1(r^2) r^2 \| \ps \|^2 + 4\left( C_2(r^2) r^2 + C_1(r^2)\right) \rh^2 \\
%&=  g_1(r) \| \ps \|^2 + g_2(r) \rh^2\,,
%\end{align*}
The pullback metric is then 
\begin{align*}\tag{b}
\bar G=\left((\Ph\i)^\ast \tilde G\right) 
&= 4C_1(r^2) r^2 \langle d\ph,d\ph\rangle + 4\left( C_2(r^2) r^2 + C_1(r^2)\right) dr^2 
\\
&=  g_1(r) \langle d\ph,d\ph\rangle + g_2(r) dr^2
\\
&= a(s) \langle d\ph,d\ph\rangle + ds^2 \,,
\end{align*}
where we introduced the functions 
$$ g_1(r) = 4C_1(r^2) r^2
\quad \text{ and }\quad
g_2(r) = 4\left( C_2(r^2) r^2 + C_1(r^2)\right)\,,$$
and where in the last expression  we changed the coordinate $r$ % on $\mathbb R_{>0}$
diffeomorphically to 
$$
s(r) = 2 \int_1^r \sqrt{ C_2(\rh^2) \rh^2 + C_1(\rh^2)}\,d\rh
\quad\text{ and let }a(s) = 4C_1(r(s)^2) r(s)^2. 
$$

The resulting metric is
a radius dependent scaling of the metric on the sphere times a different radius dependent scaling of the metric on $\R_{>0}$. Note that the metric \thetag{b} (as well as the metric in the last expression of \thetag{a}) is actually well-defined on $C^{\infty}(M,\mathbb R)\setminus \{0\} \cong \mathbb R_{>0}\x S \cap C^\infty$; this leads to a (partial) geodesic completion of $(\Dens_+(M), G)$.

Geodesics for the metric \thetag{b} follow great circles on the sphere with some time dependent stretching, since reflection at any hyperplane containing this great circle is an isometry. 

We derive the geodesic equation. Let 
$[0,1]\x(-\ep,\ep)\ni (t,s)\mapsto (r(t,s),\ph(t,s))$ be a smooth variation with fixed ends of a curve $(r(t,0),\ph(t,0))$.
The energy of the curve and its derivative with respect to the variation parameter $s$ are as follows, where $\nabla^S$ is the covariant derivative on the sphere $S$.
\begin{align*}
E(r,\ph) &=  
\int_0^1 \left( \frac 12  g_1(r) \langle \ph_t, \ph_t \rangle
+ \frac 12 g_2(r). r_t^2 \right) dt \\
\p_s E(r,\ph) &= 
\int_0^1 \Big( \frac 12  g_1'(r).r_s \langle \ph_t, \ph_t \rangle + 
 g_1(r) \langle \nabla^S_{\p_s} \ph_t, \ph_t \rangle + \\
&\qquad\qquad
+\frac 12 g_2'(r).r_s.r_t^2 + g_2(r).r_t.r_{ts} \Big) dt \\
&= \int_0^1 \Big( \frac 12  g_1'(r).r_s \langle \ph_t, \ph_t \rangle
-  g_1'(r).r_t \langle \ph_s, \ph_t \rangle 
-  g_1(r) \langle \ph_s, \nabla^S_{\p_t} \ph_t \rangle + \\
&\qquad\qquad
+\frac 12 g_2'(r).r_s.r_t^2 - g_2'(r).r_t^2.r_s - g_2(r).r_{tt}.r_s \Big) dt \\
&= \int_0^1 \Big( \frac 12  g_1'(r) \langle \ph_t, \ph_t \rangle
- \frac 12 g_2'(r).r_t^2 - g_2(r).r_{tt} \Big) r_s \\
&\qquad\qquad
- \Big(  g_1'(r).r_t \langle \ph_s, \ph_t \rangle 
+  g_1(r) \langle \ph_s, \nabla^S_{\p_t} \ph_t \rangle \Big) dt\,.
\end{align*}
Thus the geodesic equation is
\begin{equation*}\tag{c}\boxed{\;
\begin{aligned}
\nabla^S_{\p_t}\ph_{t} &= 
- \p_t\left( \log  g_1(r) \right) \ph_t
\\
r_{tt} &=
\frac 12 \frac{ g_1'(r)}{g_2(r)} \langle \ph_t, \ph_t \rangle
-\frac 12 \p_t \left( \log g_2(r) \right) r_t
\end{aligned}
\;}\end{equation*}
Using the first equation we get:
\begin{align*}
\p_t\langle \ph_t,\ph_t\rangle 
&= 2\langle\nabla_{\p_t}\ph_t,\ph_t\rangle 
= -2\, \p_t\left( \log  g_1(r) \right) \langle \ph_t, \ph_t \rangle
\\
\p_t \left( \log \langle \ph_t, \ph_t \rangle \right)
&= -2\, \p_t\left( \log g_1(r) \right)
\\
\log(\|\ph_t\|^2) &= -2 \log g_1(r) + 2\log A_0 
\quad\text{ with }\quad
A_0 = g_1(r)\, \| \ph_t \| \,,
\end{align*}
which describes the speed of $\ph(t)$ along the great circle in terms of $r(t)$; note that the quantity $g_1(r) \| \ph_t\|$ is constant in $t$. The geodesic equation \thetag{c} simplifies to 
\begin{equation*}\tag{d}\boxed{\;
\begin{aligned}
\nabla^S_{\p_t}\ph_{t} &= 
-\p_t\left( \log g_1(r) \right) \ph_t
\\
r_{tt} &= \frac {A_0^2}2 \frac{g_1'(r)}{g_1(r)^2 g_2(r)}
-\frac 12 \p_t \left( \log g_2(r) \right) r_t
\end{aligned}
\;}\end{equation*}
with $g_1(r) = 4C_1(r^2) r^2$ and 
$g_2(r) = 4\left( C_2(r^2) r^2 + C_1(r^2)\right)$.

We can solve equation \thetag{d} for $\ph$ explicitely. Given initial conditions $\ph_0, \ps_0$, the geodesic $\tilde \ph(t)$ on the sphere with radius $1$ satisfying $\tilde \ph(0) = \ph_0$, $\tilde \ph_t(0) = \ps_0$ is
\[
\tilde\ph(t) = \cos(\| \ps_0\| t) \ph_0 + \sin(\|\ps_0\| t ) \frac{\ps_0}{\|\ps_0\|}\,.
\]
We are looking for a reparametrization $\ph(t) = \tilde \ph(\al(t))$. Inserting this into the geodesic equation we obtain
\begin{align*}
\p_t^2 \left(\tilde \ph(\al) \right)
- \left\langle \p_t^2 \left(\tilde \ph(\al)\right), 
\frac{\tilde \ph(\al)}{\| \tilde \ph(\al) \|} \right\rangle \tilde \ph(\al)
&= -\p_t \left(\log g_1(r) \right) \p_t \left(\tilde \ph(\al)\right) \\
\left(\nabla^S_{\p_t}\tilde \ph_t\right)(\al) \al_t^2 + \tilde\ph_t(\al) \al_{tt}
- \left\langle \tilde\ph_t(\al) \al_{tt}, 
\frac{\tilde \ph(\al)}{\| \tilde \ph(\al) \|} \right\rangle \tilde \ph(\al)
&= -\p_t \left(\log g_1(r) \right) \tilde\ph_t(\al) \al_t \\
\al_{tt} = \p_t \left(\log g_1(r) \right) \al_t\,.
\end{align*}
With intial conditions $\al(0) = 0$ and $\al_t(0) = 1$ this equation has the solution
\[
\al(t) ={g_1(r_0)} \int_0^{t} \frac{1}{g_1(r(\ta))} \,d\ta\,,
\]
where $r_0 = r(0)$ is the initial condition for the $r$-component of the geodesic.

If the metric is written in the form $\bar G = ds^2 + a(s)\langle d\ph,d\ph\rangle$,   equation \thetag{d} becomes
\begin{equation*}
s_{tt} = \frac {A_0^2}2 \frac{a'(s)}{a(s)^2}\,,\quad
\text{ for  } A_0 = a(s) \|\ph_t\|\,, 
\end{equation*}
where $\ph(t)$ is given explicitly as above. This can be integrated into the form
\begin{equation*}\tag{e}\boxed{\;
s_{t}^2 = -\frac {A_0^2}{a(s)} + A_1\,,\quad A_1 \text{ a constant.}
\;}\end{equation*}

\subsection{Relation to hypersurfaces of revolution} \label{revolution}
We consider the metric $\bar G$
on $(W_-,W_+)\x S\cap  C^\infty$ where
$\bar G_{s,\ph}=a(s) \langle d\ph, d\ph \rangle + ds^2$ and
$a(s) = 4C_1(r(s)^2) r(s)^2$. 
Then the map $\Ps$ is an isometric embedding (remember $\langle \ph,d\ph\rangle=0$ on 
$S\cap  C^\infty$),
\begin{align*}
&\Ps: ((W_-,W_+)\x S\cap  C^\infty, \bar G) \to \big(\mathbb R\x C^{\infty}(M,\mathbb R), 
du^2 + \langle df,df\rangle\big)\,, 
\\&
\Ps(s,\ph) = \Big(\int_0^s \sqrt{1 -\frac{ a'(\si)^2}{4 a(\si)}}\,d\si\;,\; \sqrt{ a(s)}\ph\Big)\,,
\end{align*}
In fact it is defined and smooth only on the open subset  
$$\left\{(s,\ph)\in (W_-,W_+)\x S\cap  C^{\infty}: a'(s)^2 < 4 a(s)\right\}\,.$$
We will see in Sect.~\ref{curvature} that the condition $a'(s)^2 < 4 a(s)$ is equivalent to a sign condition on the sectional curvature; to be precise
\[
a'(s)^2 < 4 a(s) \Leftrightarrow \on{Sec}_{(s,\ph)}(\on{span}(X,Y)) > 0\,,
\]
where $X,Y \in T_\ph S$ is any $\bar G$-orthonormal pair of tangent vectors.
Fix some $\ph_0\in S\cap C^\infty$ and consider the generating curve 
\[
\ga(s) = \Big(\int_0^s \sqrt{1 -\frac{ a'(\si)^2}{4 a(\si)}}\,d\si\;,\; \sqrt{ a(s)} \ph_0\Big)\in \R \x C^\infty(M,\R)\,;
\]
then $\ga(s)$ is already arc-length parametrized!

Any arc-length parameterized curve $I\ni s\mapsto (c_1(s),c_2(s))$ in $\mathbb R^2$ generates a hypersurface of revolution  
$$\{(c_1(s), c_2(s)\ph): s\in I, \ph\in S\cap C^\infty\}\subset \mathbb R\x C^\infty(M,\mathbb R)\,,$$
and the induced metric in the $(s,\ph)$-parameterization is  $c_2(s)^2\langle d\ph,d\ph\rangle + ds^2$.

This suggests that the moduli space of hypersurfaces of revolution is naturally embedded in the moduli space of all metrics of the form $a(s) \langle d\ph, d\ph \rangle + ds^2$. 
Let us make this more precise in an example: In the case of $S=S^1$ and the tractrix $(c_1,c_2)$, the surface of revolution is the pseudosphere (curvature $-1$) whose universal cover is only part of the hyperbolic plane. But in polar coordinates we get a space whose universal cover is the whole hyperbolic plane. In detail: the arc-length parametrization of the tractrix and the induced metric are
\begin{align*}
&c_1(s) =  \int_0^s\sqrt{1-e^{-2\si}}\,d\si = \on{Arcosh}\big({e^{s}}\big) - \sqrt{1-e^{-2s}}
,\quad c_2(s) = e^{-s},\quad s> 0
\\
&a(s)\,d\ph^2 + ds^2 = e^{-2s}d\ph^2 + ds^2, \qquad s\in \mathbb R\,.
\end{align*}

\subsection{Completeness}
\label{completeness}

In this section we assume that $(W_-, W_+) = (-\infty, +\infty)$, which is a necessary and sufficient condition for completeness. First we have the following estimate for the geodesic distance $\on{dist}$ of the metric $\bar G$, which is valid on bounded metric balls. Let $\on{dist}_{S}$ denote the geodesic distance on $S$ with respect to the standard metric.

\begin{lemma*}
\label{lem:dist_est}
Let $(W_-, W_+) = (-\infty, +\infty)$, $(s_0, \ph_0) \in \R \x S$ and $R>0$. Then there exists $C>0$, such that
\begin{multline*}
 C\inv\left( \on{dist}_S(\ph_1, \ph_2) + |s_1 - s_2| \right)
\leq \on{dist}\left((s_1, \ph_1), (s_2, \ph_2)\right) \leq \\
 \leq C \left( \on{dist}_S(\ph_1, \ph_2) + |s_1 - s_2| \right)\,,
\end{multline*}
holds for all $(s_i, \ph_i)$ with $\on{dist}\left((s_0, \ph_0), (s_i, \ph_i)\right) < R$, $i=1,2$.
\end{lemma*}

\begin{proof}
First we observe that
\[
|s_1 - s_2| \leq \int_0^1 |s_t(t)| \,dt
\leq \int_0^1 \sqrt{ a(s) \| \ph_t \|^2 + s_t^2 } \,dt = \on{Len}(s, \ph)\,,
\]
and hence by taking the infimum over all paths,
\[
|s_1 - s_2| \leq \on{dist}\left((s_1, \ph_1), (s_2, \ph_2)\right) < 2R\,.
\]
Thus $s$ is bounded on bounded geodesic balls.

Now let $(s_i, \ph_i)$ be chosen according to the assumptions and let $(s(t), \ph(t))$ be a path connecting $(s_1, \ph_1)$ and $(s_2, \ph_2)$ with $\on{Len}(s,\ph) < 2 \on{dist}\left((s_1, \ph_1), (s_2, \ph_2)\right)$. Then for $t \in [0,1]$, 
\[
\on{dist}\left((s_0, \ph_0), (s(t), \ph(t))\right)
\leq \on{dist}\left((s_0, \ph_0), (s_1, \ph_1)\right)
+ 2\on{dist}\left((s_1, \ph_1), (s_2, \ph_2)\right)
\leq 5 R\,.
\]
In particular the path remains in a bounded geodesic ball.

Thus there exists a constant $C > 1$, such that $C\inv \leq a(s) \leq C$ holds along $(s(t),\ph(t))$. From there we obtain
\[
C\inv \int_0^1 \| \ph_t \|^2 + s_t^2 \,dt
\leq \int_0^1 a(s) \| \ph_t \|^2 + s_t^2 \,dt
\leq C \int_0^1 \| \ph_t \|^2 + s_t^2 \,dt\,,
\]
and by taking the infimum over paths connecting $(s_1, \ph_1)$ and $(s_2, \ph_2)$ the desired result follows.
\end{proof}

\begin{proposition*} If $(W_-, W_+) = (-\infty, +\infty)$,
the space $(\R \x S, \bar G)$ is metrically and geodesically complete. The subspace $(\R\x S \cap C^\infty, \bar G)$ is geodesically complete.
\end{proposition*}

\begin{proof}
Given a Cauchy sequence $(s_n, \ph_n)_{n \in \mathbb N}$ in  $\R \x S$ with respect to the geodesic distance, the lemma shows that $(s_n)_{n \in \mathbb N}$ and $(\ph_n)_{n \in \mathbb N}$ are Cauchy sequences in $\R$ and $S$ respectively. Hence they have limits $s$ and $\ph$ and by the lemma  the sequence $(s_n, \ph_n)_{n \in \mathbb N}$ converges to $(s,\ph)$ in the geodesic distance as well. It is shown in \cite[Prop.~6.5]{Lang1999} that a metrically complete, strong Riemannian manifold is geodesically complete.

Since the $\ph$-part of a geodesic in $\R \x S$ is a reparametrization of a great circle, if the initial conditions lie in $\R \x S \cap C^\infty$, so will the whole geodesic. Hence $\R \x S \cap C^\infty$ is geodesically complete.
\end{proof}

The map $W \x \on{Id} \circ \Ph : L^2(M,\R) \setminus \{ 0 \} \to \R \x S$ is a diffeomorphism and an isometry with respect to the metrics $\tilde G$ and $\bar G$.

\begin{corollary*} If $(W_-, W_+) = (-\infty, +\infty)$, 
the space $(L^2(M,\R) \setminus \{ 0\}, \tilde G)$ is metrically and geodesically complete. The subset $(C^\infty(M,\R) \setminus \{0\}, \tilde G)$ is geodesically complete.
\end{corollary*}

It remains to consider the existence of minimal geodesics.

\begin{theorem*} If $(W_-, W_+) = (-\infty, +\infty)$, then 
any two points $(s_0,\ph_0)$ and $(s_1,\ph_1)$ in $\mathbb R\x S$ can be joined by a minimal geodesic.
If $\ph_0$ and $\ph_1$ lie in $S\cap C^\infty$, then the minimal geodesic also lies in $\mathbb R\x S\cap C^\infty$. 
\end{theorem*}

\begin{proof}
If $\ph_0$ and $\ph_1$ are linearly independent, we consider the 2-space $V=V(\ph_0,\ph_1)$ spanned by $\ph_0$ and $\ph_1$ in $L^2$. Then $\mathbb R\x V\cap S$  is totally geodesic since it is the fixed point set of the isometry $(s,\ph)\mapsto (s, \mathfrak s_V(\ph))$ where $\mathfrak s_V$ is the orthogonal reflection at $V$. Thus there is exists a minimizing geodesic between $(s_0,\ph_0)$ and $(s_1,\ph_1)$ in the complete 3-dimensional Riemannian submanifold $\mathbb R\x V\cap S$.   
This geodesic is also length-minimizing in the strong Hilbert manifold $\mathbb R\x S$ by the following argument: 

Given any smooth curve $c = (s,\ph):[0,1]\to \mathbb R\x S$ between these two points, 
there is a subdivision $0=t_0<t_1<\dots<t_N=1$ 
such that the piecewise geodesic $c_1$ which first runs along  a geodesic from $c(t_0)$ to $c(t_1)$, then to $c(t_2)$, \dots, and finally to $c(t_N)$, has length 
$\on{Len}(c_1)\le \on{Len}(c)$. 
This piecewise geodesic now lies in the totally geodesic $(N+2)$-dimensional  submanifold 
$\mathbb R\x V(\ph(t_0),\dots,\ph(t_N))\cap S$. Thus there exists a geodesic  $c_2$ between the two points $(s_0,\ph_0)$ and $(s_1,\ph_1)$ which is length-minimizing in this $(N+2)$-dimensional submanifold. 
Therefore $\on{Len}(c_2)\le \on{Len}(c_1)\le \on{Len}(c)$. Moreover, $c_2=(s\o c_2,\ph\o c_2)$ lies in $\mathbb R\x V(\ph_0, (\ph\o c_2)'(0))\cap S$ which also contains $\ph_1$, thus $c_2$ lies in 
$\mathbb R\x V(\ph_0,\ph_1)\cap S$.  

If $\ph_0=\ph_1$, then $\mathbb R\x \{\ph_0\}$ is a minimal geodesic. If $\ph_0=-\ph_0$ we choose a great circle between them which lies in a 2-space $V$ and proceed as above. When $\ph_0, \ph_1 \in C^\infty$, then the 3-dimensional submanifold $\R\x V \cap S$ lies in $\R \x S\cap C^\infty$ and hence so does the minimal geodesic.
\end{proof}

%In Sect.~\ref{geodesics} we saw that each geodesics traces out a great circle in the sphere with adapted parametrization. This great circle is unique unless the two points $(u_1,f_1)$ and $(u_2,f_2)$ project to antipodal points on the sphere when we can choose one of them.

%If $f_1$ and $f_2$ project to the same point on the sphere,  the meridian $\on{Im}(\Ps)\cap V(f_1)$ is a minimal geodesic. If $f_1$ and $f_2$ project to antipodal points, we choose a great circle to get a 3-space $V(f_1,g)$ for the argument above.

%If the piecewise geodesic has a proper angle $\ne \pi$ at $c(t_1)$ then it can be shortened by going directly from $c(t_0)$ to $c(t_2)$. Thus the length of the piecewise geodesic is larger then $\tilde c$. Thus $\on{Len}(c)\ge \on{Len}(\tilde c)-\ep$. Since $\ep>0$ was arbitrary, $\tilde c$ is length minimizing. 

%we may first move it slightly such that it avoids 
%$V^\bot\cap \on{Im}(\Ps)$ with nearly the same length. Then we can write $c(t) = c(t)^V + c(t)^{V^\bot}$ as the sum of the orthogonal projections of $c(t)$ into $V$ and into $V^\bot$, and $c(t)^V$ is never 0.   Then 
%$c(t) = \cos(\th(t)).c_1(1) + c(t)^{V^\bot}$ for a smooth curve $c_1$ in $V\cap \on{Im}(\Ps)$ between the two points, and a smooth curve $c_2$ in $V^\bot$. Then $\on{Len}(c_1)\le \on{Len}(c)$ with equality if and only if $c(t)^{V^\bot}=0$. So $\on{Len}(c)\ge \on{Len}(c_1)\ge \on{Len}(\tilde c)$.

\subsection{Some geodesic completions}\label{completions}
The relation to hypersurfaces of revolution in Sect.~\ref{revolution} suggests that there are functions $C_1$ and $C_2$ such that geodesic incompleteness of the metric $G$ is 
due to a `coordinate singularity' at $W_-$ or at $W_+$. Let us write $I=(W_-,W_+)$.
We work in polar coordinates on the infinite-dimensional manifold $I \x (S\cap C^\infty)$ 
with the metric $\bar G = ds^2 + a(s)\langle d\ph,d\ph\rangle$.

\emph{Example.} For $I=(0,\infty)$ the metric $ds^2 + s^2\langle d\ph,d\ph\rangle$ describes the flat space 
$C^\infty(M,\mathbb R)\setminus \{0\}$ with the $L^2$-metric in polar coordinates. Putting $0$ back in geodesically completes the space. 

Moreover, for $\be\in (0,\pi/4]$ the metric $ds^2 + \sin^2(\be) s^2 \langle d\ph,d\ph\rangle$ describes the cone  with radial opening angle $\be$. Putting in 0 generates a tip;  sectional curvature  is a delta distribution at the tip of size $2(1-\sin(\be))\pi$. This is an orbifold with symmetry group $\mathbb Z/k\mathbb Z$ at the tip if 
$\sin(\be)=1/k$.

More generally, $ds^2 + K^2 s^2 \langle d\ph,d\ph\rangle$ describes the generalized cone whose `angle defect' at the tip is $2\pi (1-K)$; there is negative curvature at the tip if $K>1$ in which case we cannot describe it as a surface of revolution.

\emph{Example.} For $I= (0,\pi)$, the metric $ds^2 + \sin^2(s)\langle d\ph,d\ph\rangle$ describes the infinite-dimensional round sphere `of 1 dimension higher' with equator $S\cap C^\infty$ and with north- and south-pole omitted. 
This can be seen from the formula for sectional curvature from Sect.~\ref{curvature} below, or by transforming it to the hypersurface of revolution according to Sect.~\ref{revolution}. Putting back the two poles gives the geodesic completion. 
To realize this on the space of densities, we may choose a smooth and positive function $g_2(r)$ freely, and then put
\begin{align*}
g_1(r) &= \sin^2\Big(\int_1^r g_2(\rh)^{1/2}\,d\rh\Big)\,,
\quad
C_1(m) = \frac{g_1(\sqrt{m})}{4m}\,, 
\\
C_2(m) &= \frac1{4m} g_2(\sqrt{m}) - \frac1{4m^2} g_1(\sqrt{m})\,.
\end{align*}
Choosing $g_2(r)= 4r^2$ we get $g_1(r)=\sin^2(r^2-1)$ so that $C_1(m)=\frac1{4m}\sin^2(m-1)$
and $C_2(m)= 1- \frac1{4m^2}\sin^2(m-1)$ .

The general situation can be summarized in the following result:
\begin{theorem*}
If $W_->-\infty$ and if $C_1$ and $C_2$ have smooth extensions to $[0,\infty)$ and $C_1(0) > 0$, then the metric $\bar G$ has a smooth 1-point geodesic completion at $r=0$ (or $s=W_-$).

If $W_+<\infty$ and if $C_1$ and $C_2$ have smooth extensions to $(0,\infty]$ in the coordinate $1/m$, then the metric $\bar G$ has a smooth 1-point geodesic completion at $r=\infty$ (in the coordinate $1/r$), or at $s=W_+$.
\end{theorem*}

\begin{proof}
From the formulas in Sect.~\ref{overview} we get
\begin{multline*}
\Ph^*(g_1(r)\langle d\ph,d\ph\rangle + g_2(r)dr^2) 
= \frac{g_1(\|f\|)}{\|f\|^2}\langle df,df\rangle +\Big(\frac{g_2(\|f\|)}{\|f\|^2} - \frac{g_1(\|f\|)}{\|f\|^4}\Big)\langle f,df\rangle^2
\\
= 4C_1(\|f\|^2)\langle df,df\rangle + 4C_2(\|f\|^2)\langle f,df\rangle^2\,.
\end{multline*}
By a classical theorem of Whitney the even smooth functions $h(r)$ are exactly the smooth functions of $r^2$. So the metric extends smoothly at 0 to $C^{\infty}(M,\mathbb R)$.
The proof for the case $W_+ <\infty$  is similar.
\end{proof}

\subsection{Covariant derivative and curvature}\label{curvature}
In this section we will write $I = (W_-,W_+)$. In order to calculate the covariant derivative we consider the infinite-dimensional manifold $I \x S$ 
with the metric 
$\bar G = ds^2 + a(s)\langle d\ph,d\ph\rangle$ and smooth vector fields $f(s,\ph)\p_s + X(s,\ph)$ where 
$X(s,\;)\in \X(S)$ is a smooth vector field on the Hilbert sphere $S$.
We denote by $\nabla^S$ the covariant derivative on $S$ and get
\begin{align*}
 \p_s &\bar G\big(g\p_s+ Y,h\p_s + Z\big) = \p_s\big(gh + a\langle Y,Z\rangle\big) = 
\\&
= g_s h + g h_s + a_s\langle Y,Z\rangle + a\langle Y_s,Z\rangle + a\langle Y,Z_s\rangle
\\&
= \bar G\big( g_s\p_s + \frac{a_s}{2a}Y + Y_s, h\p_s + Z\big) 
+  \bar G\big(g\p_s + Y, h_s\p_s + \frac{a_s}{2a}Z + Z_s\big)
\\
X &\bar G\big(g\p_s+ Y,h\p_s + Z\big) = X\big(gh + a\langle Y,Z\rangle\big)
\\&
= dg(X).h + g.dh(X) + a\langle \nabla^S_XY,Z\rangle + a \langle Y,\nabla^S_XZ,\rangle
\\&
= \bar G\big(dg(X)\p_s + \nabla^S_XY, h\p_s+Z\big) + \bar G\big(g\p_s+Y,dh(X)\p_s + \nabla^S_XZ\big)\,.
\end{align*}
Thus the following covariant derivative on $I \x S$, which is not the Levi-Civita covariant derivative,
\begin{align*}
\bar\nabla_{f\p_s+X}(g\p_s+Y) = f.g_s\p_s + f\frac{a_s}{2a}Y + fY_s + dg(X)\p_s + \nabla^S_XY\,,
\end{align*}
respects the metric $ds^2 + a\langle d\ph,d\ph\rangle$. But it has torsion  which is given by
\begin{align*}
\on{Tor}&(f\p_s+X,g\p_s+Y)=
\\&
=\bar\nabla_{f\p_s+X}(g\p_s+Y) - \bar\nabla_{g\p_s+Y}(f\p_s+X) - [f\p_s + X, g\p_s + Y] =
\\&
% = f.g_s\p_s + f\frac{a_s}{2a}Y + fX_s + dg(X)\p_s + \nabla^S_XY
% \\&\quad
% - g.f_s\p_s - g\frac{a_s}{2a}X - gX_s  - df(Y)\p_s- \nabla^S_YX
% \\&\quad
% - f.g_s\p_s - fY_s  - dg(X)\p_s + g.f_s.\p_s  + gX_s +  df(Y)\p_s - [X,Y]^S 
% \\&
=   \frac{a_s}{2a}(fY - gX)\,.
\end{align*}
To remove the torsion we consider the endomorphisms
\begin{gather*}
\on{Tor}_{f\p_s +X}, \on{Tor}_{f\p_s +X}^{\top} : T(I \x S)\to T(I \x S)\,,
\\
\on{Tor}_{f\p_s +X}(g\p_s+Y) =  \on{Tor}(f\p_s+X,g\p_s + Y)\,,
\\
\bar G\big(\on{Tor}_{f\p_s +X}^\top(g\p_s + Y), h\p_s + Z) = \bar G\big(g\p_s + Y,\on{Tor}_{f\p_s +X}(h\p_s+ Z)\big)
\end{gather*}
The endomorphism
\begin{align*}
&A_{f\p_s+ X}(g\p_s+Y) :=
\\& 
= \tfrac12\big(\on{Tor}(f\p_s+ X,g\p_s+Y) - \on{Tor}_{f\p_s+ X}^\top(g\p_s+Y) 
-  \on{Tor}_{g\p_s+ Y}^\top(f\p_s+X)\big)
\end{align*} 
is then $\bar G$-skew, so that
\begin{equation*}
\nabla_{f\p_s + X}(g\p_s + Y) =  \bar\nabla_{f\p_s + X}(g\p_s + Y)  - A_{f\p_s+ X}(g\p_s+Y)
\end{equation*}
still respects $\bar G$ and is now torsion free. In detail we get
\begin{align*}
 \on{Tor}_{f\p_s+ X}^\top(g\p_s+Y) &= - \frac{a_s}2\langle X,Y\rangle \p_s + \frac{a_s}{2a}fY
 \\
A_{f\p_s+ X}(g\p_s +Y) &= \frac{a_s}{2}\langle X,Y\rangle\p_s - \frac{a_s}{2a}gX\,,
\end{align*}
so that $\nabla$ is the Levi-Civita connection of $\bar G$:
$$\boxed{\;
\begin{aligned}
\nabla_{f\p_s + X}(g\p_s + Y) &= \big(f.g_s+ dg(X)- \frac{a_s}{2}\langle X,Y\rangle\big)\p_s 
\\&\quad
+ \frac{a_s}{2a}(fY  + gX)+ fY_s  + \nabla^S_XY\,.
\end{aligned}
\;}$$

For the curvature computation we assume from now on that all vector fields of the form $f\p_s+X$  have $f$ constant and $X=X(\ph)$ so that in this case
\begin{align*}
\nabla_{f\p_s + X}(g\p_s + Y) &=  -\frac{a_s}{2}\langle X,Y\rangle \p_s 
+ \frac{a_s}{2a}(fY  + gX)  + \nabla^S_XY\,,
\\
[f\p_s + X,g\p_s + Y] &= [X,Y]^S\,,
\end{align*}
in order to obtain
\begin{align*}
&\nabla_{f\p_s + X}\nabla_{g\p_s + Y}(h\p_s+Z) =
\nabla_{f\p_s + X}
\big(-\frac{a_s}{2}\langle Y,Z\rangle \p_s + \frac{a_s}{2a}(gZ  + hY)  + \nabla^S_YZ\big) 
\\&
% =\Big(-f\frac{a_{ss}}{2}\langle Y,Z\rangle-\frac{a_s}{2}\langle \nabla^S_XY,Z\rangle
% -\frac{a_s}{2}\langle Y,\nabla^S_XZ\rangle
% \\&\qquad\qquad
% -\frac{a_s}2\big\langle X, \frac{a_s}{2a}(gZ  + hY)  + \nabla^S_YZ\big\rangle
% \Big)\p_s
% \\&\quad
% + \frac{a_s}{2a}\big(f \frac{a_s}{2a}(gZ  + hY)  + f\nabla^S_YZ -\frac{a_s}{2}\langle Y,Z\rangle X\big)
% \\&\quad
% +f\big(\frac{a_s}{2a}\big)_s(gZ  + hY)  + \nabla^S_X\big( \frac{a_s}{2a}(gZ  + hY)  + \nabla^S_YZ\big)
% \\&
=\Big(-f\frac{a_{ss}}{2}\langle Y,Z\rangle-\frac{a_s}{2}\langle \nabla^S_XY,Z\rangle
-\frac{a_s}{2}\langle Y,\nabla^S_XZ\rangle
\\&\qquad\qquad
-\frac{a_s^2}{4a}g\langle X, Z\rangle  -\frac{a_s^2}{4a}h\langle X,Y\rangle  
-\frac{a_s}2\langle X, \nabla^S_YZ\rangle
\Big)\p_s
\\&\quad
+ \frac{a_s^2}{4a^2}fgZ  + \frac{a_s^2}{4a^2}fhY  + \frac{a_s}{2a}f\nabla^S_YZ -\frac{a_s^2}{4a}\langle Y,Z\rangle X
\\&\quad
+\big(\frac{a_s}{2a}\big)_sfgZ  + \big(\frac{a_s}{2a}\big)_sfhY  
+ \frac{a_s}{2a}g\nabla^S_XZ  + \frac{a_s}{2a}h\nabla^S_XY  + \nabla^S_X\nabla^S_YZ
\\
% -&\nabla_{g\p_s + Y}\nabla_{f\p_s + X}(h\p_s+Z) =
% \\&
% =\Big(g\frac{a_{ss}}{2}\langle X,Z\rangle+\frac{a_s}{2}\langle \nabla^S_YX,Z\rangle
% +\frac{a_s}{2}\langle X,\nabla^S_YZ\rangle
% \\&\qquad\qquad
% +\frac{a_s^2}{4a}f\langle Y, Z\rangle  +\frac{a_s^2}{4a}h\langle Y,X\rangle  
% +\frac{a_s}2\langle Y, \nabla^S_XZ\rangle
% \Big)\p_s
% \\&\quad
% - \frac{a_s^2}{4a^2}gfZ  - \frac{a_s^2}{4a^2}ghX  - \frac{a_s}{2a}g\nabla^S_XZ +\frac{a_s^2}{4a}\langle X,Z\rangle Y
% \\&\quad
% -\big(\frac{a_s}{2a}\big)_sgfZ  - \big(\frac{a_s}{2a}\big)_sghX  
% - \frac{a_s}{2a}f\nabla^S_YZ  - \frac{a_s}{2a}h\nabla^S_YX  - \nabla^S_Y\nabla^S_XZ
% \\
-&\nabla_{[f\p_s + X,g\p_s + Y]}(h\p_s+Z) = -\nabla_{[X,Y]^S}(h\p_s+Z)
\\&
= +\frac{a_s}{2}\langle [X,Y]^S,Z\rangle\p_s - \frac{a_s}{2a}h[X,Y]^S - \nabla^S_{[X,Y]^S}Z
\end{align*}
Summing up we obtain the curvature (for general vector fields, since curvature is of tensorial character)
\begin{align*}
&\mathcal R(f\p_s + X,g\p_s + Y)(h\p_s + Z) =
\\&
= \big(\frac{a_{ss}}{2} - \frac{a_s^2}{4a}\big)\langle gX-fY,Z\rangle\p_s
+ \mathcal R^S(X,Y)Z 
\\&\quad
- \big(\big(\frac{a_s}{2a}\big)_s + \frac{a_s^2}{4a^2}\big)h(gX-fY)
+\frac{a_s^2}{4a}\big(\langle X,Z\rangle Y - \langle Y,Z\rangle X\big)\,.
\end{align*}
and the numerator for sectional curvature
\begin{align*}
&\bar G\big(\mathcal R(f\p_s + X,g\p_s + Y)(g\p_s + Y),f\p_s  +X\big)=
 a\langle\mathcal R^S(X,Y)Y,X\rangle 
\\&\quad
-\big(\frac{a_{ss}}{2} - \frac{a_s^2}{4a}\big)\big\langle gX-fY ,gX-fY\big\rangle
+\frac{a_s^2}{4}\big(\langle X,Y\rangle^2 - \langle Y,Y\rangle \langle X,X\rangle\big)\,.
\end{align*}
Let us take $X,Y\in T_\ph S$ with $\langle X,Y\rangle=0$ and 
$\langle X,X\rangle=\langle Y,Y\rangle = 1/a(s)$, then 
\begin{align*}
\on{Sec}_{(s,\ph)}(\on{span}(X,Y)) &= \frac1a - \frac{a_s^2}{4a^2}\,, &
\on{Sec}_{(s,\ph)}(\on{span}(\p_s,Y)) &= -\frac{a_{ss}}{2a} + \frac{a_s^2}{4a^2}\,,
\end{align*}
are all the possible sectional curvatures. 
Compare this with the formulae for the principal curvatures of a hypersurface of revolution in  \cite{CollHarrison13} and with the formulas for rotationally symmetric Riemannian metrics in \cite[Sect.~3.2.3]{Petersen2006}.

\subsection{Example}
The simplest case is the choice $C_1(\la) = \frac 1\la$ and $C_2(\la) = 0$. The Riemannian metric is
\[
G_\mu(\al, \be) = \frac{1}{\mu(M)} \int_M \frac \al\mu \frac \be\mu \mu\,.
\]
Then $g_1(r) = 4$ and $g_2(r) = \frac 4{r^2}$. 
This metric is geodesically complete on $C^\infty(M,\mathbb R)\setminus \{0\}$.
The geodesic equation \thetag{d} simplifies to
\[
r_{tt} = \frac{r_t^2}{r}\,.
\]
This ODE can be solved explicitely and the solution is given by
\[
r(t) = r(0) \exp \left( \frac{r_t(0)}{r(0)} t \right)\,.
\]
The reparamterization map is $\al(t) = t$ and thus the geodesic
\[
\ph(t) = \cos\left(\| \ph_t(0)\| t\right) + \sin\left(\| \ph_t(0)\| t\right)
\frac{\ph_t(0)}{\| \ph_t(0)\|}\,,
\]
describes a great circle on the sphere with the standard parametrization. Note that geodesics with $r_t(0) = 0$ are closed with period $2\pi / \| \ph_t(0)\|$. The spiraling behaviour of the geodesics can be seen in Fig.~\ref{fig:sim}.

\begin{figure}
\centering
	\includegraphics[width=.4\textwidth]{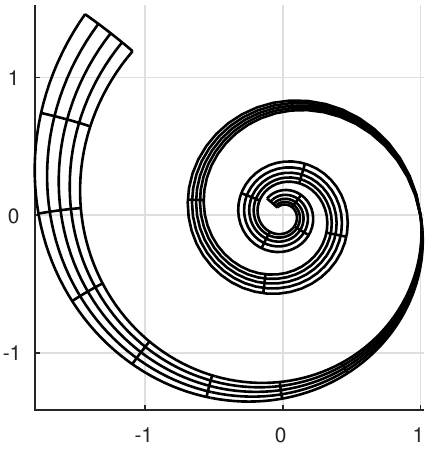}
	\includegraphics[width=.4\textwidth]{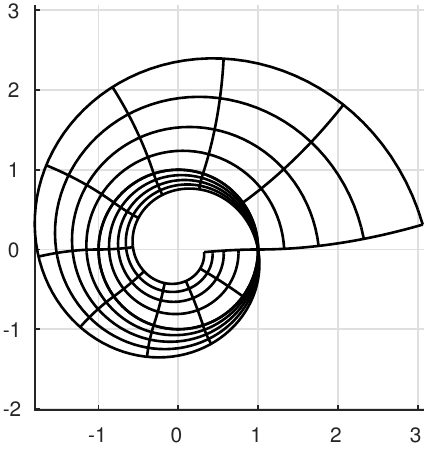}	
	\caption{Fixing $\ph(0)$, $\ph_t(0)$ with $\|\ph_t(0)\|=1$, the figure shows geodesics $r(t).\ph(t)$ starting at $r(0)=1$ for various choices of $r_t(0)$; the geodesics are shown in the orthonormal basis $\{\ph(0), \ph_t(0)\}$. A periodic geodesic can be seen on the right. The coefficients in the metric are $C_1(\la) = \la\inv$ and $C_2(\la) = 0$.}
	\label{fig:sim}
\end{figure}

\subsection{Example.} By setting $C_1(\la) = 1$ and $C_2(\la) = 0$ we obtain the Fisher--Rao metric on the space of all densities. The Riemannian metric is
\[
G_\mu(\al, \be) = \int_M \frac \al\mu \frac \be\mu \mu\,.
\]
In this case $g_1(r) = 4 r^2$ and $g_2(r) = 4$. 
The metric is incomplete towards 0 on $C^\infty(M,\mathbb R)\setminus \{0\}$. The pullback metric \thetag{b} is
\[
\tilde G = 4 r^2 \langle d\ph, d\ph \rangle + 4 dr^2\,,
\] 
and hence geodesics are straight lines in $C^\infty(M,\R) \setminus \{0\}$. In terms of the variables $(r,\ph)$, the geodesic equation \thetag{d} for $r$ is
\[
r_{tt} = \frac{A_0^2}{16} \frac 1{r^3}\,,
\]
with $A_0 = 4 r(0)^2 \| \ph_t(0)\|$. 

\begin{figure}
\centering
	\includegraphics[width=.4\textwidth]{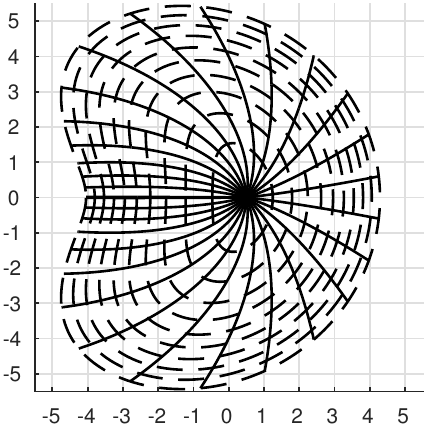}
	\includegraphics[width=.4\textwidth]{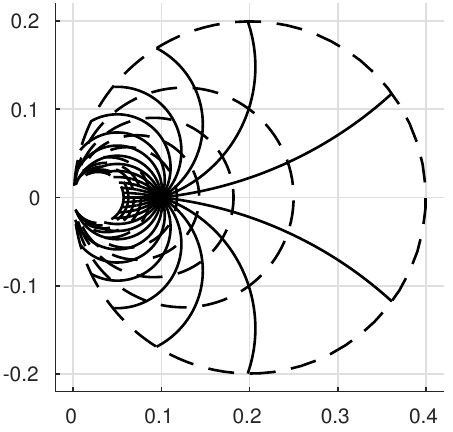}	
	\caption{Fixing $\ph(0)$, $\ph_t(0)$ with $\|\ph_t(0)\|=1$, the figure shows geodesics $r(t).\ph(t)$ for various choices of $r_t(0)$; on the left the extended Fisher--Rao metric with $C_1=C_2=1$ with geodesics starting from $r(0)=1$; on the right the metric with $C_1=\frac 1{r^2}$ with geodesics starting from $r(0)=0.1$.}
	\label{fig:ext_inf}
\end{figure}

\subsection{Example} Setting $C_1(\la) = 1$ and $C_2(\la) = 1$ we obtain the extended metric
\[
G_\mu(\al, \be) = \int_M \frac \al\mu \frac \be\mu \mu + \int_M \al \int_M \be\,.
\]
In this case $g_1(r) = 4r^2$ and $g_2(r) = 4r^2 + 4$. The geodesic equation \thetag{d} is
\[
r_{tt} = \frac{A_0^2 - 16 r^4 r_t^2}{16 r^3 \left( 1+ r^2 \right)}\,.
\]
The metric on $C^\infty(M,\mathbb R)\setminus \{0\}$ is incomplete towards 0. Geodesics for the metric can be seen in Fig.~\ref{fig:ext_inf}. Note that only the geodesic going straight into the origin seems to be incomplete.

\subsection{Example} Setting $C_1(\la) = \frac{1}{\la^2}$ and $C_2(\la) = 0$ we obtain the metric
\[
G_\mu(\al, \be) = \frac{1}{\mu(M)^2} \int_M \frac \al\mu \frac \be\mu \mu\,,
\]
which is complete towards 0, but incomplete towards infinity on $C^\infty(M,\mathbb R)\setminus \{0\}$. 
We have $g_1(r) = 4/r^2$ and $g_2(r) = 4 / r^2$. The geodesic equation \thetag{d} is
\[
r_{tt} = \frac{2 r_t^2 - A_0^2 r^6}{16 r}\,.
\]
Examples of geodesics can be seen in Fig.~\ref{fig:ext_inf}. Note that the geodesic ball extends more towards infinity than towards the origin.

%\bibliographystyle{abbrv}
%\bibliography{Fisher-Rao}%preprints,../../ref/articles}

\end{document}